\documentclass[11pt]{article}

\usepackage{graphicx}
    \graphicspath{ {figures/} }
    \usepackage[margin=0.1cm]{caption}
    \usepackage{float}
\usepackage{amsmath}
\usepackage{amsthm}
\usepackage{amssymb}
\usepackage[mathscr]{eucal}
\usepackage[all]{xy}
\usepackage{epsfig}
\usepackage{epstopdf}
\usepackage{amscd}
\usepackage[normalem]{ulem}
\usepackage{enumitem}
    \setlist[enumerate,1]{label=\textnormal{(\alph*)}}
\usepackage{thmtools, thm-restate}

\usepackage[left=1.5in,right=1.5in,top=1.5in,bottom=1.5in,
            ]{geometry}
\usepackage{hyperref}

\theoremstyle{plain}
\newtheorem{theorem}{Theorem}

\newtheorem{lemma}[theorem]{Lemma}

\theoremstyle{remark}

\theoremstyle{definition}
\newtheorem*{definition}{Definition} 

\newtheorem*{question*}{Question}

\newtheorem{observation}{Observation}

\newcommand{\ka}[2]{$#1\textnormal{a}_{#2}$}

\title{Space-Efficient Prime Knot 7-Mosaics}
\author{Aaron Heap and Natalie LaCourt}

\begin{document}

\maketitle
\begin{abstract} The concepts of tile number and space-efficiency for knot mosaics were first explored by Heap and Knowles in \cite{Heap1}, where they determined the possible tile numbers and space-efficient layouts for every prime knot with mosaic number 6 or less. In this paper, we extend those results to prime knots with mosaic number 7.
\end{abstract}


\section{Introduction}

Knot mosaics were first introduced by Lomonaco and Kauffman in \cite{Lom-Kauff} as a basic building block of blueprints for constructing an actual physical quantum system, with a mosaic knot representing a quantum knot. The mosaic system they developed consisted of creating a square array of tiles selected from the list of tiles given in Figure \ref{tiles}. These mosaic tiles are identified, respectively, as $T_0$, $T_1$, $T_2$, $\ldots$, $T_{10}$.

\begin{figure}[h]
  \centering
  \includegraphics{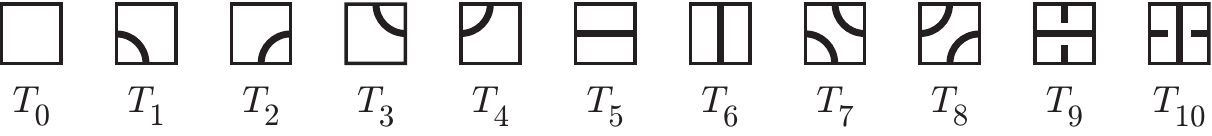}\\
  \caption{Tiles used for constructing mosaic knots.}
  \label{tiles}
\end{figure}

The first mosaic tile, $T_0$, is a blank tile, and the remaining mosaic tiles, referred to as \textit{non-blank tiles}, depict pieces of curves that will by used to construct knots or links when appropriately connected. These non-blank tiles consist of single arcs, horizontal or vertical line segments, double arcs, and over/under knot projection crossings. A \textit{connection point} of a tile is an endpoint of a curve drawn on the tile. A tile is \textit{suitably connected} if each of its connection points touches a connection point of an adjacent tile.

\begin{definition}
An $n \times n$ array of suitably connected tiles is called an \textit{$n \times n$ knot mosaic}, or \textit{$n$-mosaic}.
\end{definition}

Note that an $n$-mosaic could represent a knot or a link, as illustrated in Fig. \ref{mosaic-example}. The first two mosaics depicted are $4$-mosaics, and the third one is a $5$-mosaic. In this paper, we will be working only with knots, not links.

\begin{figure}[h]
  \centering
  \includegraphics{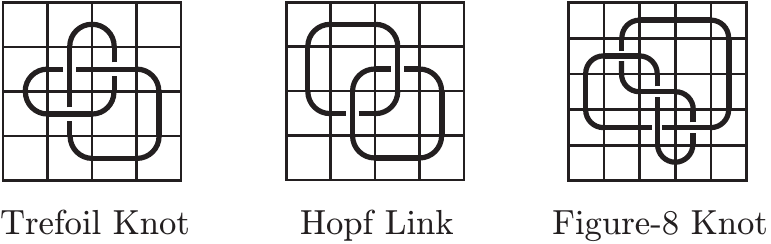}\\
  \caption{Examples of knot mosaics.}
  \label{mosaic-example}
\end{figure}

In addition to the original eleven tiles $T_0$ - $T_{10}$, we will also make use of \textit{nondeterministic tiles}, such as those in Figure \ref{nondeterministic}, when there are multiple options for a specific tile location. For example, if a tile location must contain a crossing tile $T_9$ or $T_{10}$ but we have not yet chosen which, we will use the nondeterministic crossing tile, shown as the first tile in Figure \ref{nondeterministic}. Similarly, if we know that a tile location must have four connection points but we do not know if the tile is a double arc tile ($T_7$ or $T_{8}$) or a crossing tile ($T_9$ or $T_{10}$), we will indicate this with a tile that has four connection points, as seen in the second tile of Figure \ref{nondeterministic}. If the tile contains dashed lines or arcs, these will indicate the options for that tile. The third tile in Figure \ref{nondeterministic} could be a horizontal segment $T_5$ or a single arc $T_2$.
\begin{figure}[h]
  \centering
  \includegraphics{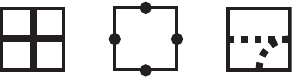}\\
  \caption{Examples of nondeterministic tiles.}
  \label{nondeterministic}
\end{figure}

There are a few knot invariants of primary importance in this paper. The \emph{crossing number} of a knot is the least number of crossings in any projection of the knot. The remaining invariants are directly related to knot mosaics. The first is the mosaic number of a knot, introduced in \cite{Lom-Kauff}.
\begin{definition}
The \textit{mosaic number} of a knot $K$ is defined to be the smallest integer $m$ for which $K$ can be represented on an $m$-mosaic.
\end{definition}

The next knot invariant is the tile number of a knot, introduced by Lee, Ludwig, Paat, and Peiffer in \cite{Lee2} and first explored by Heap and Knowles in \cite{Heap1}.
\begin{definition}
The \textit{tile number} of a knot $K$ is the smallest number of non-blank tiles needed to construct $K$ on any size mosaic.
\end{definition}

We say that a knot mosaic is \textit{minimal} if it is a realization of the mosaic number of the knot depicted on it. That is, if a knot with mosaic number $m$ is depicted on an $m$-mosaic, then that mosaic is a minimal knot mosaic. It turns out that the tile number of a knot may not be realizable on a minimal mosaic. This fact was discovered by Heap and Knowles in \cite{Heap2}, where it was shown that the knot $9_{10}$ has mosaic number 6 and tile number 27, but that on a 6-mosaic 32 non-blank tiles were required. The tile number 27 was only achievable on a larger mosaic. Because of this, it is also of some interest to know how many non-blank tiles are necessary to depict a knot on a minimal mosaic, which is known as the minimal mosaic tile number of a knot, first introduced in \cite{Heap1}.

\begin{definition}
Let $m$ be the mosaic number of $K$. The \textit{minimal mosaic tile number} of $K$ is the smallest number of non-blank tiles needed to construct $K$ on an $m$-mosaic.
\end{definition}

So the knot $9_{10}$ has mosaic number 6, tile number 27, and minimal mosaic tile number 32, with the tile number achieved on a 7-mosaic. $9_{10}$ is the simplest knot for which the tile number and minimal mosaic tile number are not equal. In this paper, we give more examples of this phenomenon.

As we work with knot mosaics, we can move parts of the knot around within the mosaic via \textit{mosaic planar isotopy moves} to obtain a different knot mosaic that depicts the same knot. Two knot mosaic diagrams are of the \emph{same knot type} (or \emph{equivalent}) if we can change one to the other via a sequence of these mosaic planar isotopy moves. An examples of a mosaic planar isotopy move is given in Figure \ref{isotopy-example}, which is equivalent to a Reidemeister Type I move. If we have a mosaic that has one of these $2 \times 2$ submosaics within it, then that submosaic can be replaced by either of the other two without changing the knot type of the depicted knot. While these moves are technically tile replacements within the mosaic, they are analogous to the planar isotopy moves used to deform standard knot diagrams. A more complete list of these moves are given and discussed in \cite{Lom-Kauff} and \cite{Kuriya}. We will make significant use of these moves throughout this paper, as we attempt to construct knot mosaics that use the least number of non-blank tiles.

\begin{figure}[h]
  \centering
  \includegraphics{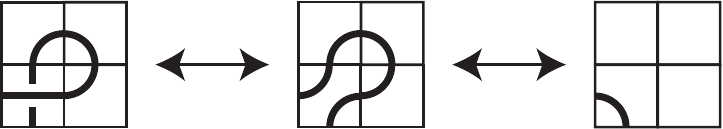}\\
  \caption{Example of a mosaic planar isotopy move.}
  \label{isotopy-example}
\end{figure}

A knot mosaic is called \textit{reduced} if there are no unnecessary, easily removed crossings in the knot mosaic diagram. One such reducible crossing is given in the first $2 \times 2$ submosaic of Figure \ref{isotopy-example}. Another example is given in Figure \ref{reducible}, where the crossing in the fourth row and third column is unnecessary. If we want to create knot mosaics efficiently, using the least number of non-blank tiles necessary, we will want to avoid these reducible crossings.

\begin{figure}[h]
  \centering
    \phantom{.} \hfill
    \begin{minipage}{.4\textwidth}
       \centering
        \includegraphics{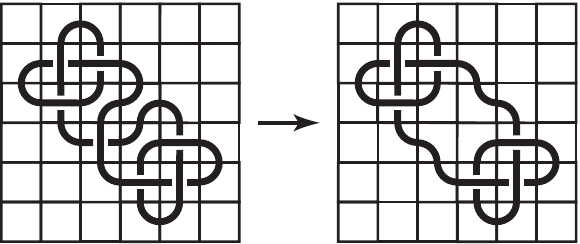}
        \caption{Reducing a reducible knot mosaic.}
        \label{reducible}
    \end{minipage} \hfill
    \begin{minipage}{0.46\textwidth}
       \centering
        \includegraphics{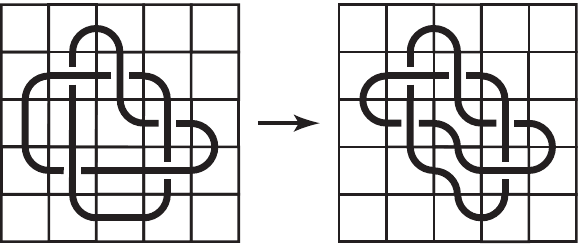}
        \caption{Space-inefficient and space-efficient mosaics of the $5_1$ knot.}
        \label{space-efficient}
    \end{minipage} \hfill

\end{figure}

\begin{definition}
A knot $n$-mosaic is \textit{space-efficient} if it is reduced and the number of non-blank tiles is as small as possible on an $n$-mosaic without changing the knot type of the depicted knot.
\end{definition}

The number of non-blank tiles in a knot mosaic that is space-efficient cannot be decreased through a sequence of mosaic planar isotopy moves. In Figure \ref{space-efficient}, the two knot mosaics depict the same knot (the $5_1$ knot). However, the first knot mosaic uses nineteen non-blank tiles and the second knot mosaic uses only seventeen. In fact, seventeen is the minimum number of non-blank tiles possible to create this knot on a 5-mosaic. Therefore, the second mosaic is space-efficient, but the first one is not.


In \cite{Heap1}, the possible layouts for space-efficient $n$-mosaics, together with the possible values of the minimal mosaic tile numbers and tile numbers, are given for all $n \leq 6$. In the supplement to \cite{Heap2}, we are provided with a table of knot mosaics that includes space-efficient mosaics for all prime knots with mosaic number 6 or less. In each of these prime knot mosaics, either the tile number or minimal mosaic tile number is realized. In this paper, we expand upon these ideas to include 7-mosaics.

For a quality introduction to knot mosaics, we refer the reader to \cite{Lee2}. For more details related to traditional knot theory, we refer the reader to \cite{Adams} by Adams. We also point out that throughout this paper we make use of KnotScape \cite{Thistle}, created by Thistlethwaite and Hoste, to verify that a given knot mosaic represents the specific knot we claim it does. Finally, special thanks are due to James Canning, who was kind enough and brilliant enough to create for us a program that automated the process of creating the mosaics in the Table of Mosaics of Section \ref{table-of-mosaics}.

\section{Space-Efficient 7-mosaics}

In \cite{Heap1}, the authors determined the bounds for the number of non-blank tiles needed for a space-efficient $n$-mosaic ($n \geq 4$) depicting a prime knot and in which either every row or every column of the mosaic is occupied. The lower bound is $5n-8$. If $n$ is even, then the upper bound is $n^2-4$. If $n$ is odd, then the upper bound is $n^2-8$. Therefore, in the specific case of $n=7$, if $t$ is the number of non-blank tiles used in the mosaic, then  $27 \leq t \leq 41$. In \cite{Heap1}, the authors provided a conjecture for the possible values for the number of non-blank tiles used in the mosaic, and we provide a confirmation of that conjecture now. (We note that the original conjecture included the extra possibility of 40 non-blank tiles, but the layout that resulted in that option was not space-efficient.)

\begin{theorem}\label{tile-numbers} If we have a space-efficient 7-mosaic of a prime knot $K$ for which either every column or every row is occupied, then the only possible values for the number of non-blank tiles used in the mosaic are 27, 29, 31, 32, 34, 36, 37, 39, and 41. Furthermore, any such mosaic of $K$ is equivalent (up to symmetry) to one of the following mosaics.
\begin{center}
  \includegraphics[width=1\columnwidth]{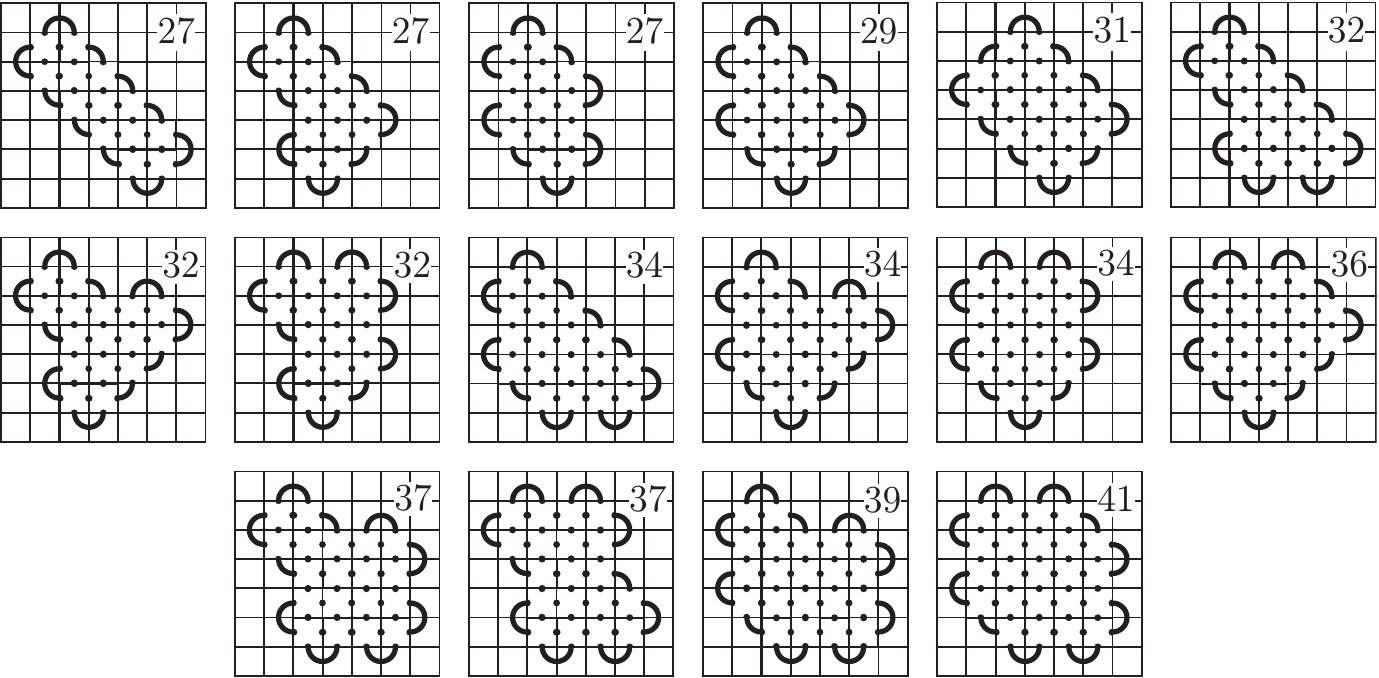}\\
\end{center}
\end{theorem}

The layouts given in Theorem \ref{tile-numbers} are listed in order of the number of non-blank tiles used, with that number displayed in the upper-right corner. As we can see from the layouts, other than the non-blank tiles that form the outer edges of the knot, the remaining non-blank tiles each have four connection points. As a consequence, the use of the horizontal and vertical segment tiles, $T_5$ and $T_6$, are not necessary for space-efficient mosaics. This does not mean that a mosaic that makes use of these tiles is not space-efficient. However, for any space-efficient mosaic that makes use of them, there is a planar isotopy move that can be applied that removes the line segment tiles without changing the number of non-blank tiles. For the sake of clarity, we postpone the proof of this theorem until Section \ref{proof-of-theorem}.

Now that we know the space-efficient layouts of prime knot 7-mosaics, we seek to determine what prime knots actually fit within these. Naively, there are four tile options for each remaining nondeterministic tile, which means that there are $4^{13}=67,108,864$ options for the first layout and $4^{23}$ (over 70.3 trillion) options for the last layout. Of course these options can be greatly reduced using symmetry and space-efficiency. Also, we know from \cite{Heap2} that every prime knot with crossing number 8 or less can fit on a 6-mosaic and has tile number 27 or less. Therefore, we only seek to find prime knots with crossing number 9 or more. For the first layout, symmetry, space-efficiency, and restricting to nine or more crossing tiles, we can reduce the number of options to the low thousands, a huge improvement over 67 million. Because of the vast number of options, especially for the larger layouts, we limit ourselves in this paper to just the three simplest layouts, the ones that use 27 non-blank tiles.

Using the techniques of \cite{Heap2}, we fill portions of these layouts with the $3 \times 3$ building blocks shown in Figure \ref{building-blocks}, which have either two, three, or four crossing tiles. (A similarly filled block with only one crossing tile is not space-efficient.) For example, in the first layout with 27 non-blank tiles, the upper-left $3 \times 3$ submosaic can be filled with any of the options given in Figure \ref{building-blocks}. The lower-right $3 \times 3$ submosaic can be filled with a rotation of these. Then we populate the rest of the layout with either double-arc tiles or crossing tiles, making sure that we use at least nine crossing tiles.

\begin{figure}[h]
  \centering
  \includegraphics{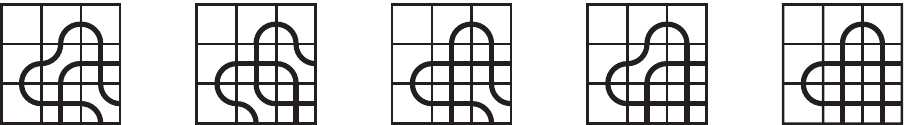}\\
  \caption{$3 \times 3$ building blocks.}
  \label{building-blocks}
\end{figure}

Once the mosaics are completely filled within the specified layout, we eliminate any links and composite knots, any duplicate layouts that are equivalent to others via obvious mosaic planar isotopy moves, and any mosaics for which the tile number can easily be reduced by a simple mosaic planar isotopy. Finally, we use KnotScape to determine what knots are depicted in the mosaic be choosing the crossings so that they are alternating, as well as all possible non-alternating combinations. Doing this for all three layouts with 27 non-blank tiles, we find all knots with mosaic number 7 and tile number 27. We also find prime knots with mosaic number 6 and minimal mosaic tile number 32 but whose tile number 27 is only realized on mosaics of size 7 or larger. The following theorems summarize the results. When listing prime knots with crossing number 10 or less, we will use the Alexander-Briggs notation, matching Rolfsen's table of knots in \emph{Knots and Links} \cite{Rolfsen}.
For knots with crossing number 11 or higher, we use the Dowker-Thistlethwaite name of the knot.
See information about this naming at \emph{KnotInfo} \cite{Knotinfo}.

\begin{theorem}\label{t27-m6} The following prime knots have mosaic number 6, minimal mosaic tile number 32, and tile number 27 realized on a space-efficient 7-mosaic: $9_{10}$, $10_{11}$, $10_{20}$, $10_{21}$, and \ka{11}{341}.
\begin{figure}[h]
    \centering
    \begin{minipage}{0.17\textwidth}
        \captionsetup{skip=3pt}
        \includegraphics{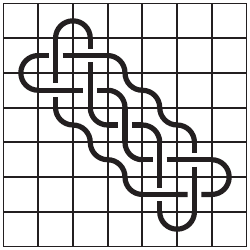}
        \caption*{$9_{10}$}
    \end{minipage} \hfill
    \begin{minipage}{0.17\textwidth}
        \captionsetup{skip=3pt}
        \includegraphics{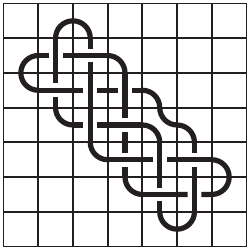}
        \caption*{$10_{11}$}
    \end{minipage} \hfill
    \begin{minipage}{0.17\textwidth}
        \captionsetup{skip=3pt}
        \includegraphics{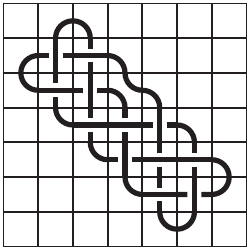}
        \caption*{$10_{20}$}
    \end{minipage} \hfill
    \begin{minipage}{0.17\textwidth}
        \captionsetup{skip=3pt}
        \includegraphics{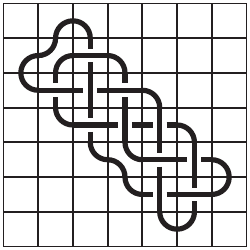}
        \caption*{$10_{21}$}
    \end{minipage} \hfill
    \begin{minipage}{0.17\textwidth}
        \captionsetup{skip=3pt}
        \includegraphics{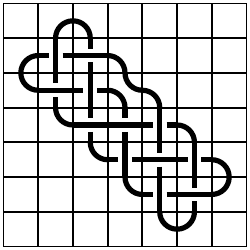}
        \caption*{\ka{11}{341}}
    \end{minipage} \hfill
\end{figure}
\end{theorem}

\begin{proof}  Each of the knots listed were found to have mosaic number 6 and minimal mosaic tile number 32 in \cite{Heap2}. In the statement of the theorem, we give space-efficient 7-mosaics with tile number 27 for each of these knots. So, in order to create these knots on a 6-mosaic, the fewest non-blank tiles possible is 32, while it is possible to depict these knots using only 27 non-blank tiles on a 7-mosaic.
\end{proof}



\begin{theorem}\label{t27-m7} The following prime knots have mosaic number 7 and tile number 27:
\begin{enumerate}
  \item $9_6$, $9_{15}$, $9_{18}$,
  \item $10_{5}$, $10_{6}$, $10_{7}$, $10_{8}$, $10_{9}$, $10_{10}$, $10_{13}$, $10_{14}$, $10_{15}$, $10_{16}$, $10_{17}$, $10_{18}$, $10_{19}$, $10_{24}$, $10_{25}$, $10_{26}$, $10_{29}$, $10_{30}$, $10_{31}$, $10_{32}$, $10_{33}$, $10_{35}$, $10_{36}$, $10_{38}$, $10_{39}$,
  \item \ka{11}{90}, \ka{11}{93}, \ka{11}{119}, \ka{11}{145}, \ka{11}{180}, \ka{11}{184}, \ka{11}{185}, \ka{11}{192}, \ka{11}{203}, \ka{11}{205}, \ka{11}{210}, \ka{11}{226}, \ka{11}{306}, \ka{11}{307}, \ka{11}{308}, \ka{11}{309}, \ka{11}{311}, \ka{11}{333}, \ka{11}{336}, \ka{11}{337}, \ka{11}{363},
  \item \ka{12}{541}, \ka{12}{601}, \ka{12}{1024}, \ka{12}{1034}, \ka{12}{1126}, and
  \item \ka{13}{4304}.
\end{enumerate}
\end{theorem}

\begin{proof}  We have given space-efficient mosaics with tile number 27 for each of these knots in the table of knots in Section \ref{table-of-mosaics}. Since the mosaic number of each of these knots is 7, we know that they cannot have a tile number smaller than 27.
\end{proof}

Finally, we point out that all of the knot mosaics for the knots listed in Theorem \ref{t27-m7} come from the first layout given in Theorem \ref{tile-numbers}. Neither of the next two layouts, also with 27 non-blank tiles, resulted in any knots that the first layout did not. This is analogous to what we see with 6-mosaics, where there are two space-efficient layouts with 22 non-blank tiles, but the second layout led to the same results as the first. See \cite{Heap2} for more information.
%

\section{Useful Observations and the Proof of Theorem \ref{tile-numbers}}\label{proof-of-theorem}

In this section, we will prove Theorem \ref{tile-numbers}. As we progress toward this goal, we first provide some useful terms and observations that we will make use of as we attempt to create space-efficient mosaics and count the number of non-blank tiles needed to create them.


Suppose there are two adjacent single arc tiles, $T_1$, $T_2$, $T_3$, or $T_4$, that share a connection point, and the other connection points enter the same adjacent row or column. The four options are shown in Figure \ref{caps}, and we will refer to these collectively as \emph{caps} and individually as \emph{top caps}, \emph{right caps}, \emph{bottom caps}, and \emph{left caps}, respectively. We will encounter mosaics that have pieces that are similar to caps but with line segment tiles, $T_5$ or $T_6$, between the arc tiles or pieces that can be easily changed to these via planar isotopy moves. We will call these \emph{reducible caps} and define them to be a collection of suitably connected tiles with no crossing tiles that are, or can be changed via planar isotopy moves to, two single arc tiles that are connected by line segment tiles and whose other connection points enter the same adjacent row or column. Essentially, a reducible cap can meander around but can easily be reduced to a reducible cap that does not meander, such as the last example in Figure \ref{caps}, which could be simplified to the fourth example.

\begin{figure}[h]
  \centering
  \includegraphics{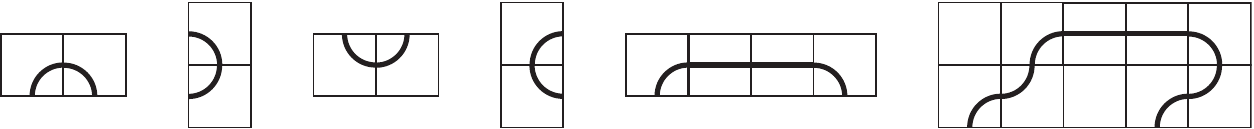}\\
  \caption{Caps and reducible caps.}
  \label{caps}
\end{figure}

Equipped with this terminology, we consider the following observations that will assist us in counting the minimum number of non-blank tiles necessary to create knot mosaics. More general versions of the first three observations were proven in \cite{Heap1}, and we include them here, without proof, specifically applied to prime knots. The fourth observation is a mosaic planar isotopy that we make use of several times that can reduce the number of non-blank tiles used in a mosaic. The fifth observation is simply a reminder of how to recognize when a mosaic does not represent a prime knot.

This first observation tells us that we can create all of our space-efficient knot mosaics without using the corner tile locations, the first and last tiles in the first row and last row of the mosaic. Because the the outer rows and columns need not be occupied, we may assume that the first tile and the last tile in the first occupied row and column is a blank tile, and similarly for the last occupied row and column.

\begin{observation}(Lemma 3.1, \cite{Heap1})\label{no-corners}
We can assume that the corner tiles of any space-efficient knot mosaic are blank $T_0$ tiles. Furthermore, for a space-efficient knot mosaic, the first and last tile location of the first and last occupied row and column are blank (or can be made blank via a planar isotopy move that does not change the tile number).
\end{observation}

\begin{observation}(Lemma 3.3, \cite{Heap1})\label{four-below} Suppose we have a space-efficient knot mosaic. If there is a cap or reducible cap in any row (or column), then the two tiles that share connection points with the cap must have four connection points. Examples are shown in Figure \ref{four-below-1}.
\end{observation}

\begin{figure}[h]
  \centering
  \includegraphics{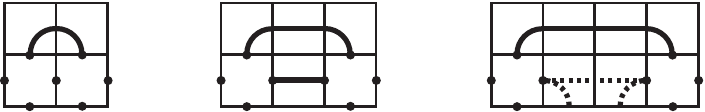}\\
  \caption{Tiles that share connection points with caps must have four connection points.}
  \label{four-below-1}
\end{figure}

\begin{observation}(Lemma 3.6, \cite{Heap1})\label{top-caps} Suppose we have a space-efficient 7-mosaic. Then the first occupied row of the mosaic can be simplified so that the only non-blank tiles form either one or two top caps. Similarly, the last occupied row is made up of bottom caps, and the first and last occupied columns are made up of left caps and right caps, respectively. Although it is not explicitly stated in \cite{Heap1}, the proof there shows that every reducible cap can be simplified to a cap without increasing the number of non-blank tiles.
\end{observation}


\begin{observation}\label{common-isotopy} The mosaic planar isotopy move given in Figure \ref{isotopy} does not change the knot type of the knot mosaic. The two tiles with four connection points rotate into the tile positions to the right, with the upper tile rotating 90 degrees counterclockwise and the lower tile rotating 90 degrees clockwise.  The number of non-blank tiles in the resulting mosaic, after applying this isotopy, is always less than or equal to the number of non-blank tiles in the original mosaic.
\begin{figure}[h]
  \centering
  \includegraphics{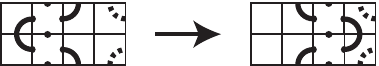}\\
  \caption{This mosaic planar isotopy move does not increase the number of non-blank tiles in a knot mosaic.}
  \label{isotopy}
\end{figure}
\end{observation}

\begin{observation}\label{composite-knot} If a knot mosaic can be separated into two nontrivial, space-efficient pieces connected by exactly two connection points, then the knot mosaic represents a composite knot. An example of this can be seen in Figure \ref{reducible}.
\end{observation}

As we have previously noted, the use of the horizontal and vertical segment tiles, $T_5$ and $T_6$, are not necessary for create space-efficient mosaics. Once we show this fact, proving Theorem \ref{tile-numbers} is simple. Observation \ref{top-caps} tells us that $T_5$ and $T_6$ tiles are not needed in the first and last occupied rows and columns. Next, we show that they are not needed in the second and penultimate occupied rows and columns.

\begin{lemma}\label{no-segments} Suppose we have a space-efficient 7-mosaic of a prime knot. Then the mosaic can be simplified so that there are no horizontal or vertical line segment tiles in the second occupied row/column and the next-to-last occupied row/column.
\end{lemma}

The proof of this lemma is quite simple but long because of the accompanying figures that depict the various cases. However, knowing that we do not need to use horizontal and vertical segment tiles in the first two and last two rows and columns greatly reduces the possibilities for the layouts of a space-efficient 7-mosaic.

\begin{proof}  We prove the lemma for the second occupied row, and the other options follow by rotational symmetry. By Observation \ref{top-caps}, the first occupied row only contains one or two top caps, with the rest of the positions filled by blank tiles. By Observation \ref{four-below}, the tile locations directly below these caps must have four connection points, and this also prevents any vertical segment tiles from occurring in the second row. The only options (up to symmetry) for the first two occupied rows are as depicted in Figure \ref{first-row-caps}. We assume that the first occupied row of the mosaic is actually the top row since, if it is not, we can simply shift all of the tiles of the mosaic upward. In each case below, we will reach a contradiction when we assume that there is a horizontal segment tile in the second row. In most cases, we construct partial mosaics for every possible tile choice until it is obvious that the mosaic is not space-efficient or does not represent a prime knot. We always base our choices on the observations stated above, knowing that the tiles sharing connection points with any cap must have four connection points and avoiding the corner tiles, reducible caps, links, composite knots, and reducible knots.

\begin{figure}[h]
  \centering
  \includegraphics{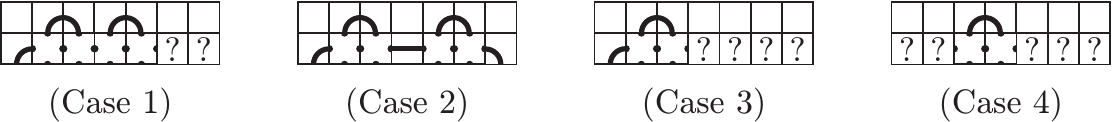}\\
  \caption{The first occupied row can only have one or two caps.}
  \label{first-row-caps}
\end{figure}

\textbf{Case 1:} First, we consider the case where there are two consecutive top caps. The first five tile positions in the second row are determined by Observation \ref{four-below}. This observation also prevents the sixth tile position from being a horizontal segment tile. Thus the sixth tile must be a single arc tile $T_1$, and the final tile position must be blank.

\textbf{Case 2:} The next case has two top caps with a blank tile in between them. In this case, the second row is completely determined and must have a horizontal segment tile as seen in Figure \ref{first-row-caps}. Our claim is that any completion of this mosaic will not be space-efficient or will not represent a prime knot. To see this, we will examine the remaining rows. There is actually only one possibility for the third row as well, which can be seen in Figure \ref{twocap-thirdrow}(a). The first and last tiles in the third row must complete a left and right cap, respectively, and the second and sixth tiles must have four connection points. The third and fifth tile positions in the third row must also have four connection points. Otherwise, they would be single arc tiles, and any resulting space-efficient mosaic would not represent a prime knot.

We now consider the remaining positions for the middle column. The tile in the fourth position can either be blank, another horizontal tile, or a single arc tile $T_1$ or $T_2$. (Because of symmetry, the $T_1$ and $T_2$ cases are equivalent.) If the entire middle column has only blank or horizontal segment tiles, then the mosaic is not space-efficient, as we can collapse this column. That is, there must eventually be a $T_1$ (or $T_2$) single arc tile. This tile clearly cannot be in the sixth row (by Observation \ref{four-below}) or the seventh row. Thus, the only options are as in Figure \ref{twocap-thirdrow}(b) and (c).

\begin{figure}[h]
  \centering
  \includegraphics{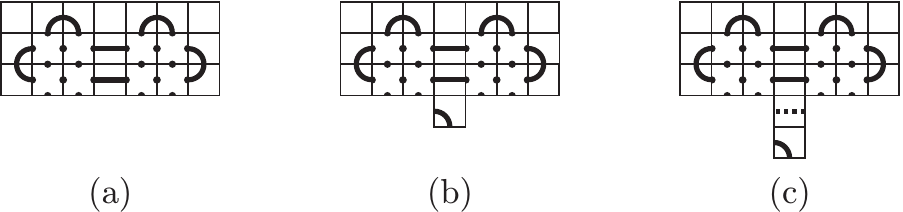}\\
  \caption{The first three rows are completely determined, and there are three options for the middle column.}
  \label{twocap-thirdrow}
\end{figure}

Suppose the tile in the middle position of the fourth row is the single arc tile $T_1$, as in Figure \ref{twocap-thirdrow}(b). Because of the locations of connection points and sides of tiles with no connection points, the options are limited. For example, the fifth position in the fourth row can only be a vertical segment tile $T_6$ or single arc tile $T_3$. Similarly, the third tile position in the fourth row can only be a single arc tile $T_3$ or a tile with four connection points. Ultimately, this leads to only six options for completing some of the surrounding tiles, and these are given in Figure \ref{twocap-thirdrow1}. It is easy to see that none of these are space-efficient. For the first three, simply slide the right-most tiles to the left by collapsing the horizontal tiles. For the fourth, fifth, and sixth mosaics, slide the upper, left quadrant to the right by collapsing the horizontal segments.

\begin{figure}[h]
  \centering
  \includegraphics[width=1\columnwidth]{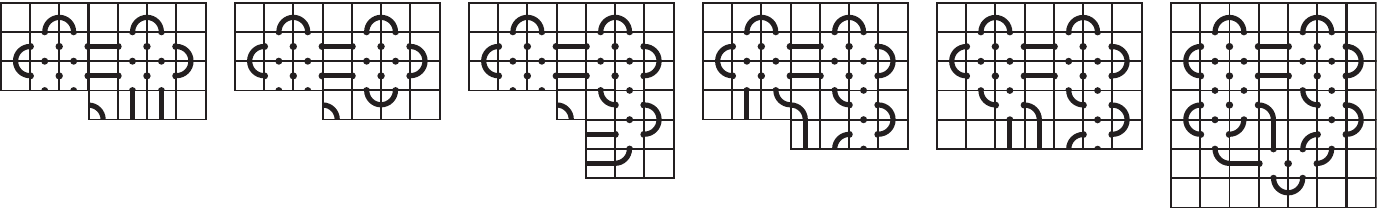}\\
  \caption{If the tile in the fourth row, fourth column is a $T_1$ tile, then the mosaic is not space-efficient.}
  \label{twocap-thirdrow1}
\end{figure}

Finally, suppose the tile in the middle position of the fourth row is either blank or a horizontal tile, and the tile in the fifth row, fourth column is the single arc tile $T_1$, as in Figure \ref{twocap-thirdrow}(c). There are only seven ways to complete the tile positions to the right of these, and we provide them in Figure \ref{twocap-thirdrow2}. In each case, it is easy to see that the resulting mosaics are not space-efficient. In the first six cases, we simply slide the upper, right quadrant of the mosaics to the left by collapsing all of the horizontal segment tiles. In the case of the seventh mosaic, we can lower the tile number of the mosaic using the planar isotopy given in Observation \ref{common-isotopy}.

\begin{figure}[h]
  \centering
  \includegraphics{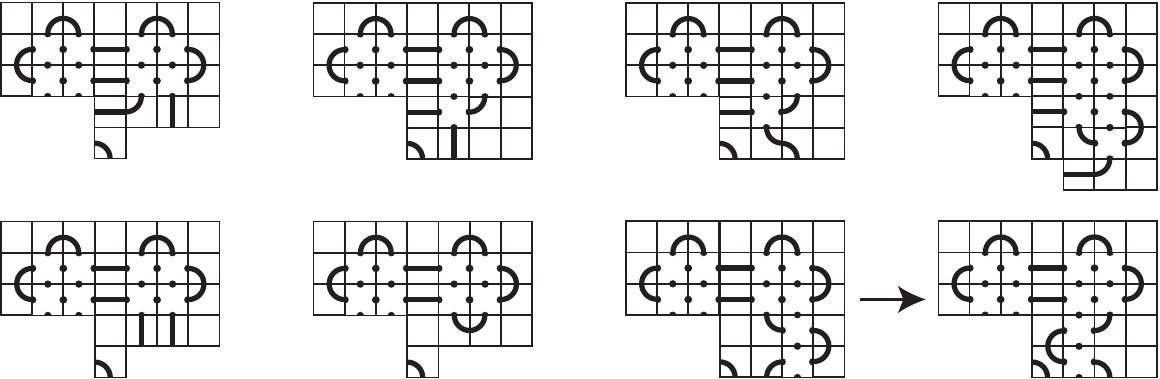}\\
  \caption{If the tile in the fourth row, fourth column is blank or a horizontal segment, then the mosaic is not space-efficient.}
  \label{twocap-thirdrow2}
\end{figure}

\textbf{Case 3:} Now let us consider the case where there is only one top cap in the first occupied row, and it is located in the first two tile positions after the corner tile, as in the third option in Figure \ref{first-row-caps}. Then the first tile in the second occupied row must be a single arc tile $T_2$, followed by two tiles with four connection points. There must also be a single arc tile $T_1$ in this row, but this $T_1$ tile cannot be part of a right cap (that is, the tile below it is not a $T_4$ tile). To see this, assume the $T_1$ tile is part of a right cap. If the $T_1$ tile is in the fourth tile position of the second row, then, using Observations \ref{four-below} and \ref{composite-knot}, it is easy to see that the knot mosaic is either not prime or not space-efficient. If the $T_1$ tile is in the fifth, sixth, or seventh position, then Observation \ref{four-below} says the preceding tile position must have four connection points, which contradicts the fact that the first row only has a single top cap.

We now examine the third row. The first tile must complete the left cap with a single arc tile $T_3$, and the second tile position must have four connection points. The third position must also have four connection points. Otherwise, this tile position would be a single arc tile $T_4$, and the mosaic would either not be space-efficient or would not represent a prime knot (Observation \ref{composite-knot}).

For the sake of contradiction, suppose there are horizontal segment tiles in the second row. There clearly cannot be three horizontal segment tiles because this forces the arc tile $T_1$ into the seventh position in this row, which is necessarily part of a right cap, and we have already ruled this out.

If there is only one horizontal segment in the second row, then it must be in the fourth tile position, and the fifth tile position is the arc tile $T_1$. We know this is not part of a right cap, and we look at the tiles below the horizontal segment. Directly below the horizontal segment must be a single arc tile $T_1$ or another horizontal segment tile. The $T_1$ option, shown in the first mosaic of Figure \ref{onecap-secondrow}, is easily simplified to reduce the number of non-blank tiles. For the horizontal segment option, the options are similar to those in Case 2. The tile below it can only be a blank, horizontal segment, or single arc ($T_1$ or $T_2$) tile. Eventually there must be a $T_1$ or $T_2$ tile to avoid the possibility of just collapsing the entire column, and this must occur in either the fourth row or the fifth row. With each of these options in mind, the remaining fourteen partially completed mosaics shown in Figure \ref{onecap-secondrow} cover all possibilities in this scenario, and it is again easy to see that none of them are space-efficient.

\begin{figure}[h]
  \centering
  \includegraphics{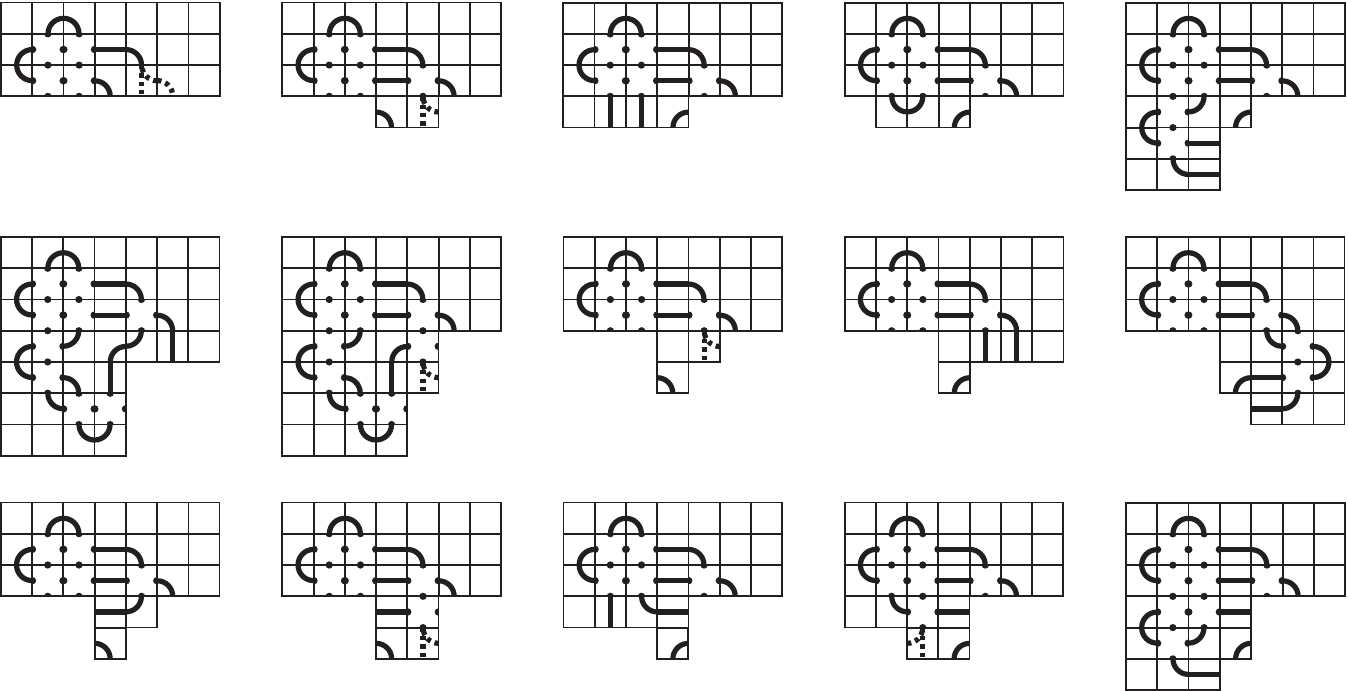}\\
  \caption{Possible configurations with a horizontal segment in the fourth tile position of the second row.}
  \label{onecap-secondrow}
\end{figure}

If there are two horizontal line segments in the second row, they must be in the fourth and fifth tile positions, and the sixth tile position is the single arc tile $T_1$. In the third row, the tile in the fourth position must be a single arc tile $T_1$ or a horizontal segment tile. If it is the $T_1$ tile, then the fifth position must be blank or the single arc tile $T_2$. In either case, shown in Figure \ref{two-hor-segs1}(a) and (b), it is easy to see that we can eliminate the horizontal segments in the second row without increasing the number of non-blank tiles.

\begin{figure}[h]
  \centering
  \includegraphics{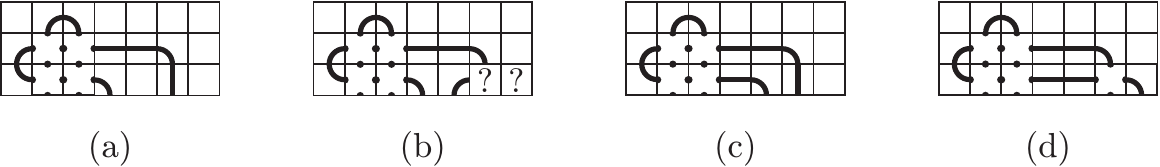}\\
  \caption{Possible configurations of the first three rows when the second row has two horizontal line segments.}
  \label{two-hor-segs1}
\end{figure}

If the tile in the fourth position of the third row is a horizontal segment tile, there are only two ways to complete the third row, and they are depicted in Figure \ref{two-hor-segs1}(c) and (d). The first one is easily seen to reduce to only one horizontal segment in the second row, which was covered above. The second possibility is also simple to eliminate after we examine a couple of tiles in the fourth row of the mosaic and perform a simple planar isotopy move. The sixth and seventh positions of the fourth row must complete the right cap, and there is an equivalent mosaic, with the fewer non-blank tiles, in which the tiles of this right cap are moved into the position of a top cap and the horizontal segments are removed, as shown in Figure \ref{two-hor-segs3}.

\begin{figure}[h]
  \centering
  \includegraphics{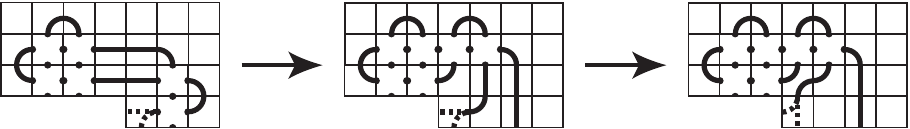}\\
  \caption{Simplifying the mosaic when the second and third row have two horizontal line segments.}
  \label{two-hor-segs3}
\end{figure}

\textbf{Case 4:} Now we consider the final case in Figure \ref{first-row-caps}, where there is a single top cap, and it is located in the third and fourth tile positions of the first occupied row. We will assume that the first column is occupied, otherwise a shift of the mosaic to the left would reduce this to Case 3. It is easy to see that a horizontal segment tile is not allowed in the second or sixth position of the second row, as this would necessarily violate Observation \ref{four-below}. Thus, if there is a horizontal segment tile, the only possibility is for it to be in the fifth position. Just as in previous cases, there must eventually be a single arc tile, $T_1$ or $T_2$, in the fifth column below the horizontal segment. There are twenty possibilities, and they are given in Figure \ref{onecap2-secondrow}. Each mosaic is either not space-efficient or does not represent a prime knot.
\begin{figure}[h]
  \centering
  \includegraphics{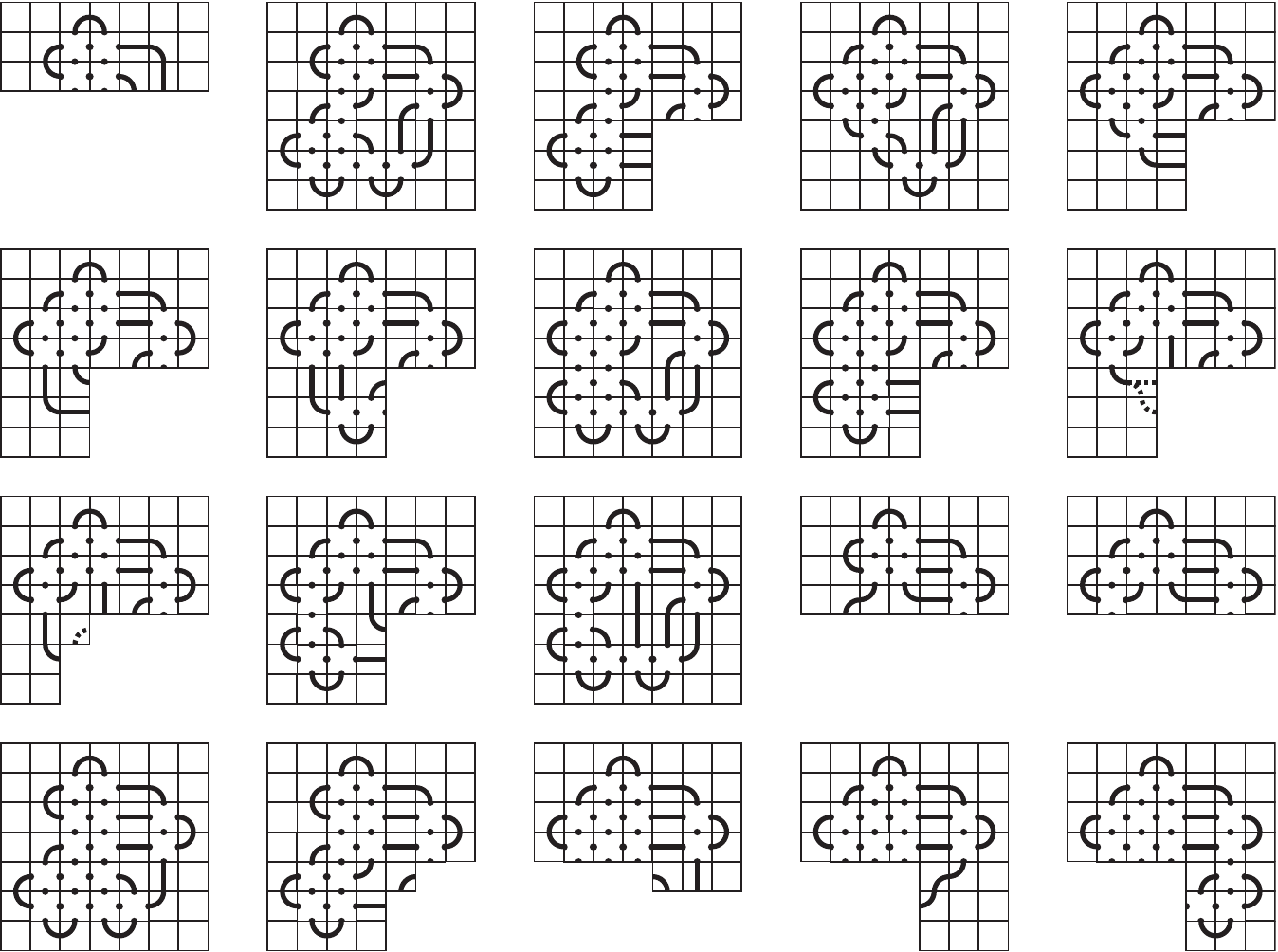}\\
  \caption{Possible configurations when there is a single top cap occupying the third and fourth position of the first row and a horizontal segment tile in the fifth position of the second row.}
  \label{onecap2-secondrow}
\end{figure}

Having completed all four cases, we have shown that any space-efficient 7-mosaic of a prime knot can be simplified so that the second occupied row does not have a horizontal segment tile.
\end{proof}

Now that we have shown that all space-efficient 7-mosaics of prime knots can be created without horizontal or vertical segment tiles in the first two and last two rows and columns, we are ready to turn our focus to proving Theorem \ref{tile-numbers}.

\begin{pfref}
We assume we have a space-efficient 7-mosaic of a prime knot for which every row is occupied. The case where every column is occupied is equivalent. Our goal is to show that the only possible layouts for the mosaic are those given in the statement of the theorem or that it is equivalent to one of them via a planar isotopy that does not increase the number of non-blank tiles. The resulting number of non-blank tiles follows immediately.

The first and seventh rows of the mosaic have either 2 or 4 non-blank tiles (one or two caps by Observation \ref{top-caps}). In either case, we have either 4 or 6 non-blank tiles in the second and sixth rows. In the proof of Lemma \ref{no-segments}, we show that there are only three options for the first row. These lead to only four possibilities for the first two rows (up to symmetry), and these are given in Figure \ref{7mosaic-layouts_row12}.

\begin{figure}[h]
  \centering
  \includegraphics{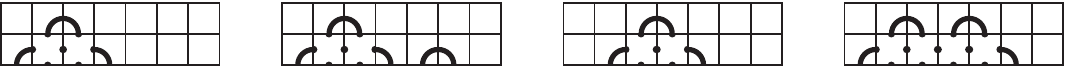}\\
  \caption{The first two rows of a space-efficient 7-mosaic.}
  \label{7mosaic-layouts_row12}
\end{figure}

Assuming the first column is occupied, the options for the first two rows extend to the first two columns. After removing any duplicates that are equivalent up to symmetry, there are nine options for the first two rows and columns of the mosaic. See Figure \ref{7mosaic-layouts_col12}.

\begin{figure}[h]
  \centering
  \includegraphics{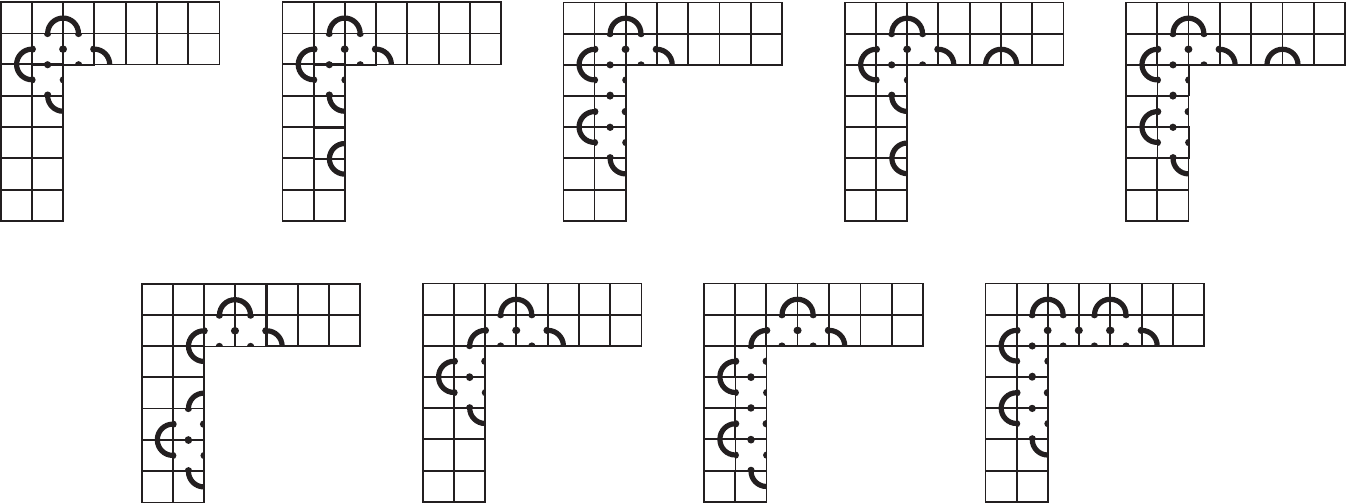}\\
  \caption{The first two rows and columns of a space-efficient 7-mosaic.}
  \label{7mosaic-layouts_col12}
\end{figure}

Next, we consider the bottom two rows and the two rightmost columns. We have assumed that all rows are occupied but not necessarily every column. However, in order to avoid composite knots or space-inefficiency, there must be at least 4 connection points between any two rows (Observation \ref{composite-knot}), except between the first two and last two rows. Removing options that are equivalent to others via symmetry, we find 20 possibilities for the two outermost rows and columns. See Figure \ref{7mosaic-layouts_outer-shell}.

\begin{figure}[h]
  \centering
  \includegraphics{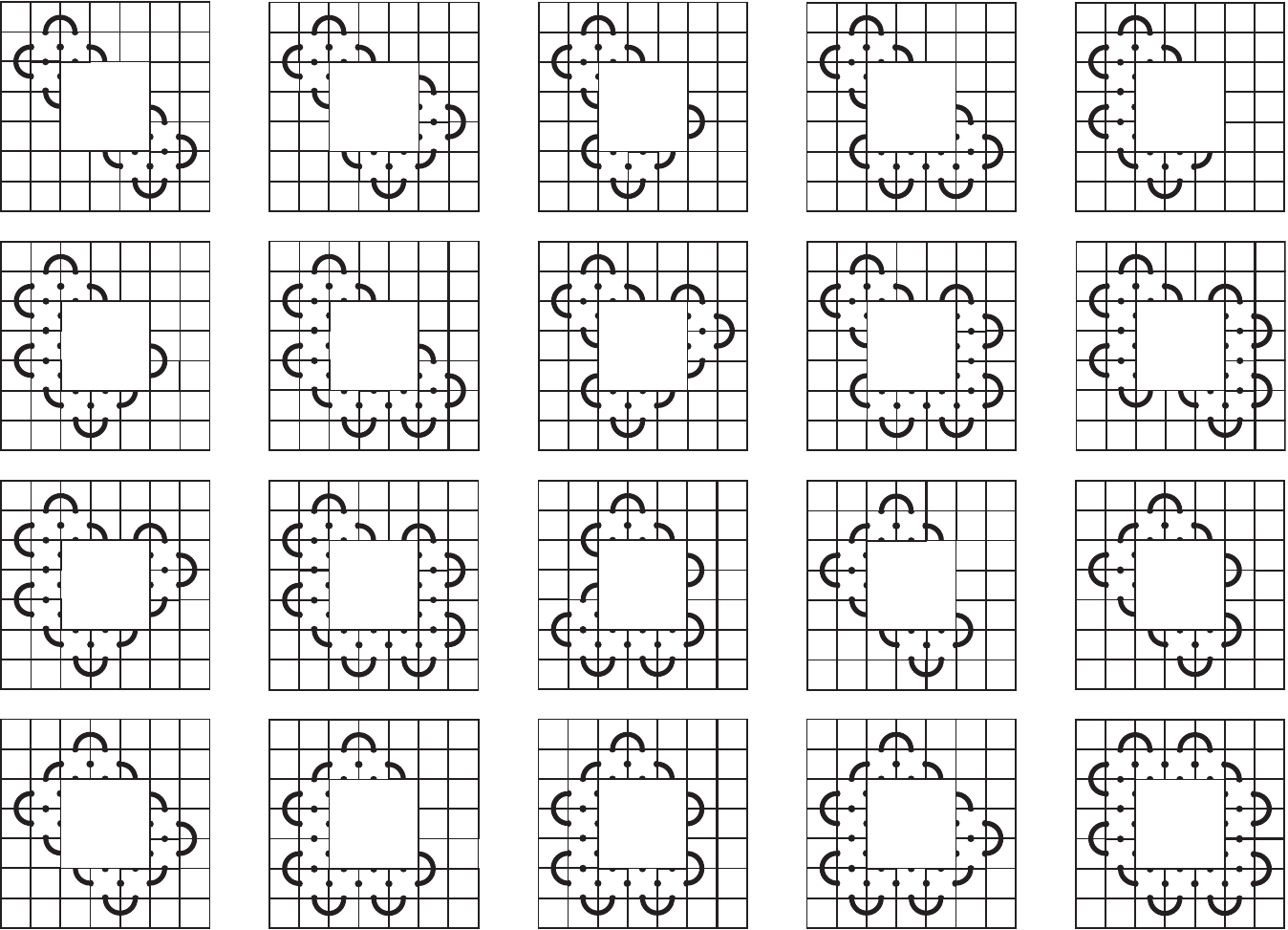}\\
  \caption{The first and last two rows and columns of a space-efficient 7-mosaic.}
  \label{7mosaic-layouts_outer-shell}
\end{figure}

%
%
%

Now that we have a complete set of possibilities for the outer shell of a space-efficient 7-mosaic, we want to fill in the inner $3 \times 3$ block. Some of that can be accomplished using Observation \ref{four-below}. Using Observation \ref{composite-knot}, we need at least four connection points between each of the middle three rows and columns, and, for example, in the upper, left $3 \times 3$ block of the first mosaic in Figure \ref{7mosaic-layouts_outer-shell}, the tile in the third row and third column cannot be a single arc tile $T_4$.
Therefore, the tile in that location must have four connection points.

These simple observations are applied as we determine the options for filling in the inner $3 \times 3$ blocks of the mosaics in Figure \ref{7mosaic-layouts_outer-shell}. For some, there is only one way to complete the mosaic, such as the first, second, third, fourth, eighth, ninth, tenth, and thirteenth mosaic. All of these, except the second one, lead to one of the desired layouts given in the statement of the theorem. The others have more than one possible way to complete them and require further consideration.

The second outer shell in Figure \ref{7mosaic-layouts_outer-shell} must be completed with horizontal or vertical segment tiles, as seen in the first mosaic of Figure \ref{7mosaic-layouts_sp-ineff1}. The resulting mosaic is clearly not space-efficient, as we can reduce the number of non-blank tiles with a planar isotopy that shifts the upper, left $3 \times 3$ block to the left.

\begin{figure}[h]
  \centering
  \includegraphics{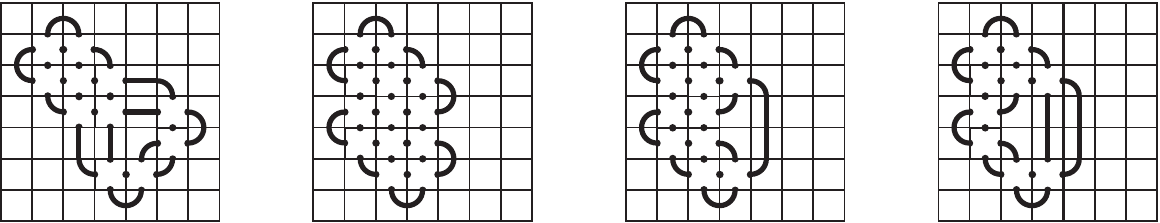}\\
  \caption{The second outer shell completed is not space-efficient, and the fifth outer shell can be completed in three ways.}
  \label{7mosaic-layouts_sp-ineff1}
\end{figure}

For each of the remaining outer shells, if we fill the remaining tile positions of the inner $3 \times 3$ block, when possible, with tiles that have four connection points, we end up with the remaining layouts given in the statement of the theorem. If we assume that at least one of those tile positions does not have four connection points, there are several possibilities to consider, which we look at next, but each one is either not space-efficient or can be changed to one of the desired layouts without changing the number of non-blank tiles. For example, if we complete the fifth outer shell with tiles that have four connection points when possible, the result is the second mosaic in Figure \ref{7mosaic-layouts_sp-ineff1}, which is the third layout given in the statement of the theorem. However, there are two alternative completions, given in the third and fourth mosaics of Figure \ref{7mosaic-layouts_sp-ineff1}. In both cases, the vertical segment tiles can be altered by a planar isotopy that changes it to the second mosaic without changing the number of non-blank tiles.

%

The sixth, seventh, eleventh, twelfth, fifteenth, eighteenth, and twentieth outer shells each have one alternative completion, which are shown in Figure \ref{7mosaic-layouts_sp-ineff2}. None of these are space-efficient as each can be simplified to one of the layouts given in the statement of the theorem using the planar isotopy shown in Observation \ref{common-isotopy}.

\begin{figure}[h]
  \centering
  \includegraphics{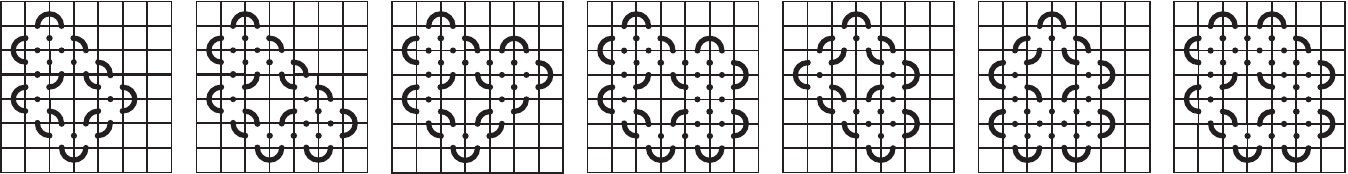}\\
  \caption{The sixth, seventh, eleventh, twelfth, fifteenth, eighteenth, and twentieth outer shells completed in ways other than those given in Theorem \ref{tile-numbers}.}
  \label{7mosaic-layouts_sp-ineff2}
\end{figure}

For the fourteenth and seventeenth outer shells, there are two alternative completions, given in Figure \ref{7mosaic-layouts_sp-ineff3}. In each case, the vertical segment tiles can be removed by a planar isotopy without changing the number of non-blank tiles.

\begin{figure}[h]
  \centering
  \includegraphics{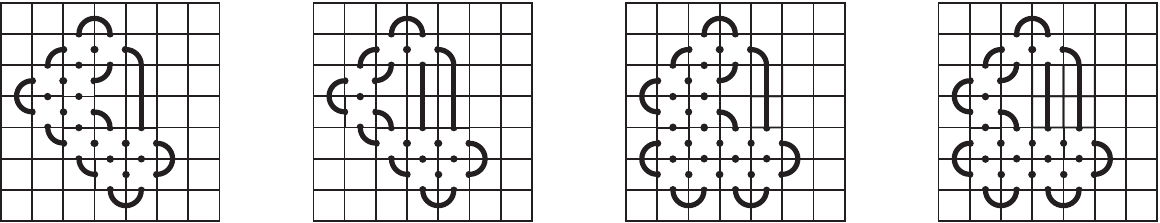}\\
  \caption{The alternative completions of the fourteenth and seventeenth outer shells can be simplified to not use the vertical segment tiles.}
  \label{7mosaic-layouts_sp-ineff3}
\end{figure}

The sixteenth outer shell has six alternative completions, and they are given in Figure \ref{7mosaic-layouts_sp-ineff4}. The first two are easily simplified, reducing the number of non-blank tiles, using the isotopy of Observation \ref{common-isotopy}. The next three are easily seen to be equivalent to the first two. The sixth alternative can be simplified, as shown, by rotating the tile in the fourth row and second column, which has four connection points, into the fourth row, third column.

\begin{figure}[h]
  \centering
  \includegraphics{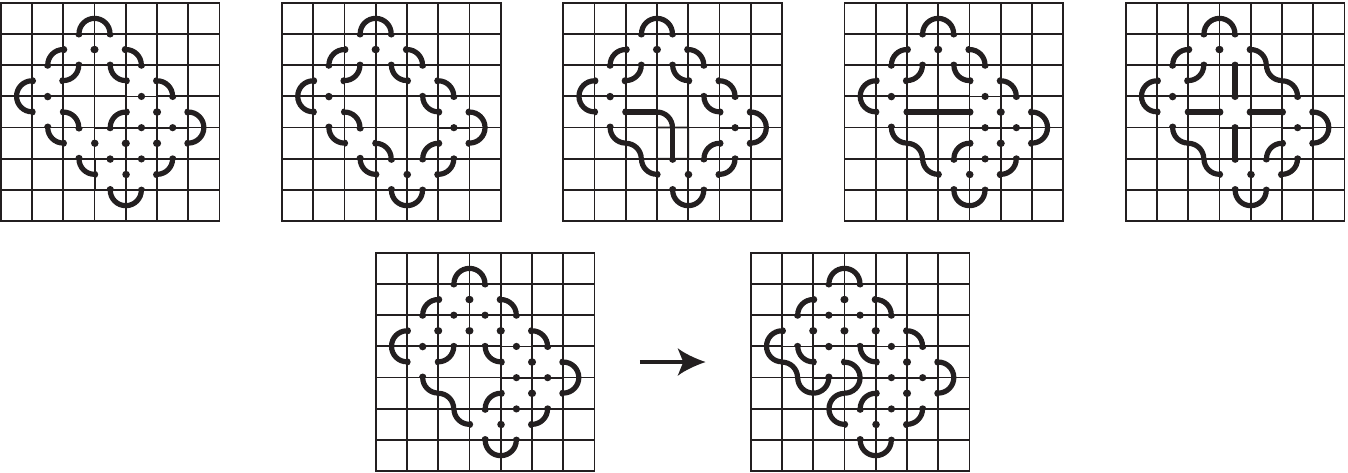}\\
  \caption{The sixteenth outer shell completed in ways other than those given in Theorem \ref{tile-numbers}.}
  \label{7mosaic-layouts_sp-ineff4}
\end{figure}

Finally, we encounter six alternative completions of the nineteenth outer shell, see in Figure \ref{7mosaic-layouts_sp-ineff5}. Each of these are handled in ways similar to the sixteenth outer shell.

\begin{figure}[h]
  \centering
  \includegraphics[width=1\columnwidth]{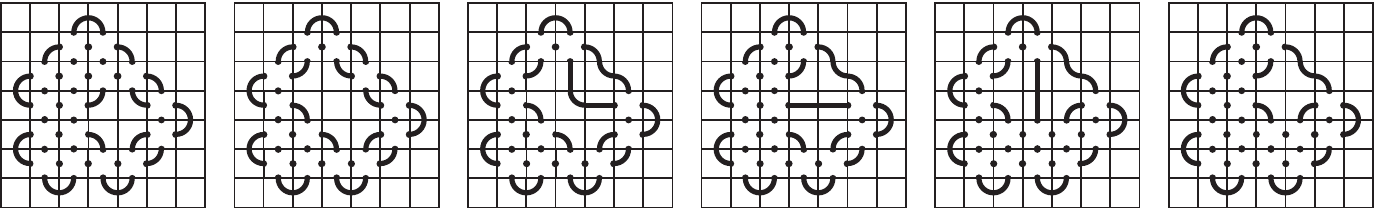}\\
  \caption{The nineteenth outer shell completed in ways other than those given in Theorem \ref{tile-numbers}.}
  \label{7mosaic-layouts_sp-ineff5}
\end{figure}

This completes our discussion of filling in the outer shells given in Figure \ref{7mosaic-layouts_outer-shell}. We have exhausted all possibilities of completing a space-efficient 7-mosaic in which all rows are occupied, arriving at those listed in the statement of the theorem.
\end{pfref}

\section{Mosaics for Theorem \ref{t27-m7}}\label{table-of-mosaics}

In this section we include the knot mosaics for each of the prime knots listed in Theorem \ref{t27-m7}. These mosaics constitute the proof for the theorem.
\begin{figure}
    \centering
    \phantom{.} \hfill
    \begin{minipage}{0.17\textwidth}
        \captionsetup{skip=3pt}
        \centering
        \includegraphics{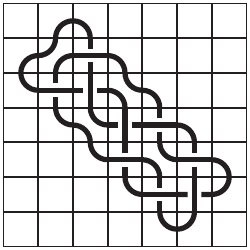}
        \caption*{$9_{6}$}
    \end{minipage} \hfill
     \begin{minipage}{0.17\textwidth}
        \captionsetup{skip=3pt}
        \centering
        \includegraphics{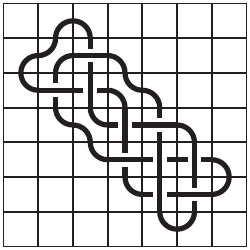}
        \caption*{$9_{15}$}
    \end{minipage} \hfill
    \begin{minipage}{0.17\textwidth}
        \captionsetup{skip=3pt}
        \centering
        \includegraphics{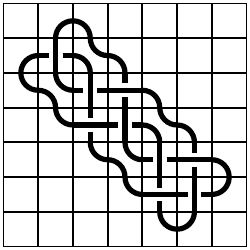}
        \caption*{$9_{18}$}
    \end{minipage} \hfill
     \begin{minipage}{0.17\textwidth}
        \captionsetup{skip=3pt}
        \centering
        \includegraphics{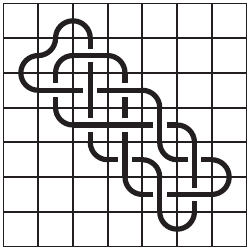}
        \caption*{$10_{5}$}
    \end{minipage} \hfill
    \begin{minipage}{0.17\textwidth}
        \captionsetup{skip=3pt}
        \centering
        \includegraphics{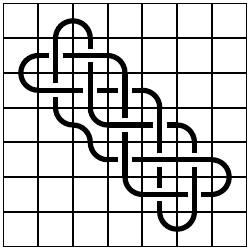}
        \caption*{$10_{6}$}
    \end{minipage} \hfill \phantom{.}
\end{figure}
\begin{figure}
     \vspace{-.1in} \centering
    \phantom{.} \hfill
    \begin{minipage}{0.17\textwidth}
        \captionsetup{skip=3pt}
        \centering
        \includegraphics{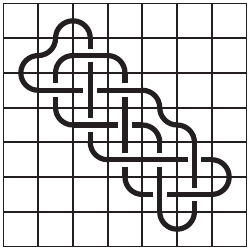}
        \caption*{$10_{7}$}
    \end{minipage} \hfill
     \begin{minipage}{0.17\textwidth}
        \captionsetup{skip=3pt}
        \centering
        \includegraphics{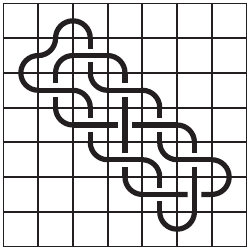}
        \caption*{$10_{8}$}
    \end{minipage} \hfill
    \begin{minipage}{0.17\textwidth}
        \captionsetup{skip=3pt}
        \centering
        \includegraphics{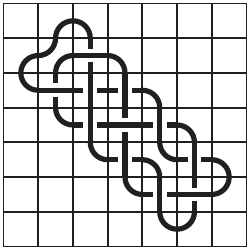}
        \caption*{$10_{9}$}
    \end{minipage} \hfill
     \begin{minipage}{0.17\textwidth}
        \captionsetup{skip=3pt}
        \centering
        \includegraphics{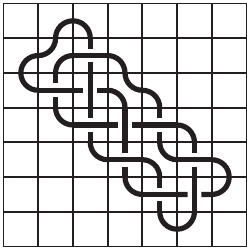}
        \caption*{$10_{10}$}
    \end{minipage} \hfill
    \begin{minipage}{0.17\textwidth}
        \captionsetup{skip=3pt}
        \centering
        \includegraphics{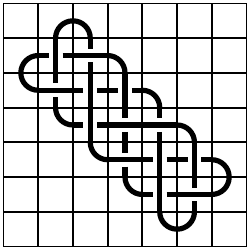}
        \caption*{$10_{13}$}
    \end{minipage} \hfill \phantom{.}
\end{figure}
\begin{figure}
    \vspace{-.1in}\centering
    \phantom{.} \hfill
    \begin{minipage}{0.17\textwidth}
        \captionsetup{skip=3pt}
        \centering
        \includegraphics{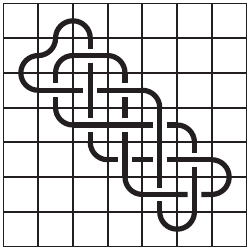}
        \caption*{$10_{14}$}
    \end{minipage} \hfill
     \begin{minipage}{0.17\textwidth}
        \captionsetup{skip=3pt}
        \centering
        \includegraphics{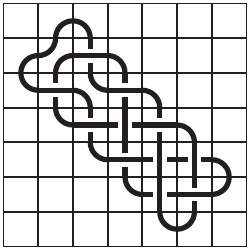}
        \caption*{$10_{15}$}
    \end{minipage} \hfill
    \begin{minipage}{0.17\textwidth}
        \captionsetup{skip=3pt}
        \centering
        \includegraphics{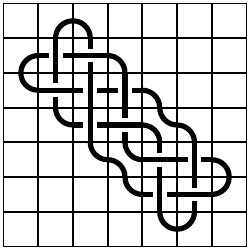}
        \caption*{$10_{16}$}
    \end{minipage} \hfill
     \begin{minipage}{0.17\textwidth}
        \captionsetup{skip=3pt}
        \centering
        \includegraphics{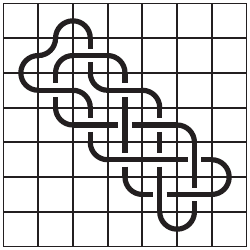}
        \caption*{$10_{17}$}
    \end{minipage} \hfill
    \begin{minipage}{0.17\textwidth}
        \captionsetup{skip=3pt}
        \centering
        \includegraphics{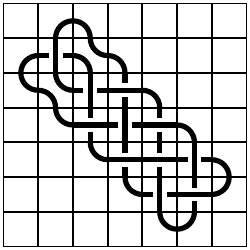}
        \caption*{$10_{18}$}
    \end{minipage} \hfill \phantom{.}
\end{figure}
\begin{figure}
     \vspace{-.1in} \centering
    \phantom{.} \hfill
    \begin{minipage}{0.17\textwidth}
        \captionsetup{skip=3pt}
        \centering
        \includegraphics{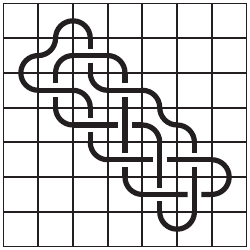}
        \caption*{$10_{19}$}
    \end{minipage} \hfill
     \begin{minipage}{0.17\textwidth}
        \captionsetup{skip=3pt}
        \centering
        \includegraphics{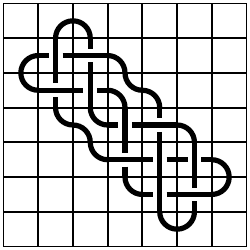}
        \caption*{$10_{24}$}
    \end{minipage} \hfill
    \begin{minipage}{0.17\textwidth}
        \captionsetup{skip=3pt}
        \centering
        \includegraphics{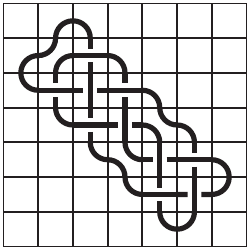}
        \caption*{$10_{25}$}
    \end{minipage} \hfill
     \begin{minipage}{0.17\textwidth}
        \captionsetup{skip=3pt}
        \centering
        \includegraphics{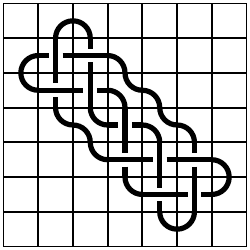}
        \caption*{$10_{26}$}
    \end{minipage} \hfill
    \begin{minipage}{0.17\textwidth}
        \captionsetup{skip=3pt}
        \centering
        \includegraphics{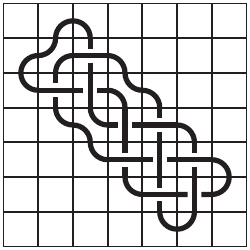}
        \caption*{$10_{29}$}
    \end{minipage} \hfill \phantom{.}
\end{figure}
\begin{figure}
     \vspace{-.1in} \centering
    \phantom{.} \hfill
    \begin{minipage}{0.17\textwidth}
        \captionsetup{skip=3pt}
        \centering
        \includegraphics{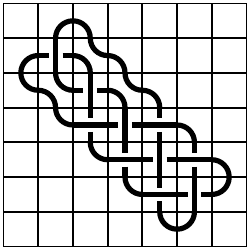}
        \caption*{$10_{30}$}
    \end{minipage} \hfill
     \begin{minipage}{0.17\textwidth}
        \captionsetup{skip=3pt}
        \centering
        \includegraphics{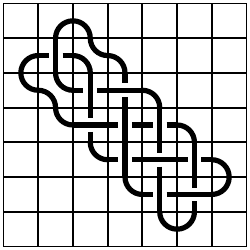}
        \caption*{$10_{31}$}
    \end{minipage} \hfill
    \begin{minipage}{0.17\textwidth}
        \captionsetup{skip=3pt}
        \centering
        \includegraphics{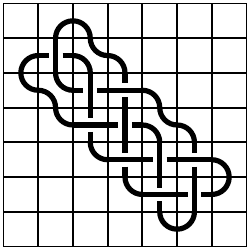}
        \caption*{$10_{32}$}
    \end{minipage} \hfill
     \begin{minipage}{0.17\textwidth}
        \captionsetup{skip=3pt}
        \centering
        \includegraphics{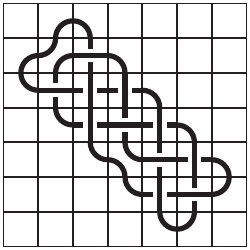}
        \caption*{$10_{33}$}
    \end{minipage} \hfill
    \begin{minipage}{0.17\textwidth}
        \captionsetup{skip=3pt}
        \centering
        \includegraphics{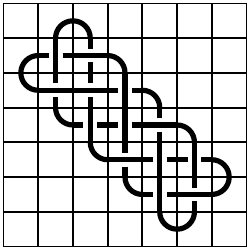}
        \caption*{$10_{35}$}
    \end{minipage} \hfill \phantom{.}
\end{figure}
\begin{figure}
     \vspace{-.1in} \centering
    \phantom{.} \hfill
    \begin{minipage}{0.17\textwidth}
        \captionsetup{skip=3pt}
        \centering
        \includegraphics{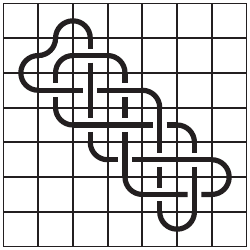}
        \caption*{$10_{36}$}
    \end{minipage} \hfill
     \begin{minipage}{0.17\textwidth}
        \captionsetup{skip=3pt}
        \centering
        \includegraphics{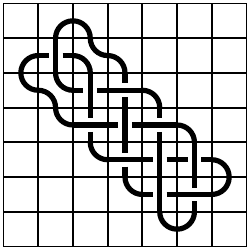}
        \caption*{$10_{38}$}
    \end{minipage} \hfill
    \begin{minipage}{0.17\textwidth}
        \captionsetup{skip=3pt}
        \centering
        \includegraphics{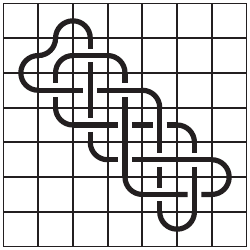}
        \caption*{$10_{39}$}
    \end{minipage} \hfill
     \begin{minipage}{0.17\textwidth}
        \captionsetup{skip=3pt}
        \centering
        \includegraphics{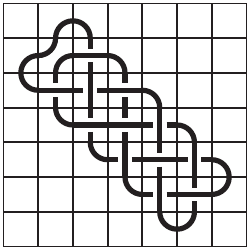}
        \caption*{\ka{11}{90}}
    \end{minipage} \hfill
    \begin{minipage}{0.17\textwidth}
        \captionsetup{skip=3pt}
        \centering
        \includegraphics{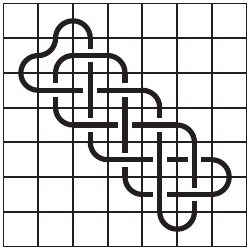}
        \caption*{\ka{11}{93}}
    \end{minipage} \hfill \phantom{.}
\end{figure}
\begin{figure}
    \vspace{-.1in} \centering
    \phantom{.} \hfill
    \begin{minipage}{0.17\textwidth}
        \captionsetup{skip=3pt}
        \centering
        \includegraphics{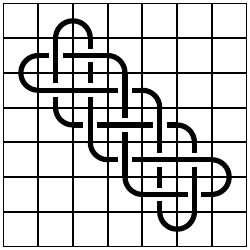}
        \caption*{\ka{11}{119}}
    \end{minipage} \hfill
     \begin{minipage}{0.17\textwidth}
        \captionsetup{skip=3pt}
        \centering
        \includegraphics{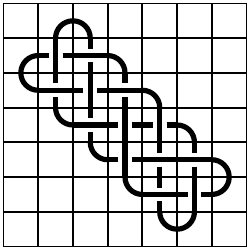}
        \caption*{\ka{11}{145}}
    \end{minipage} \hfill
    \begin{minipage}{0.17\textwidth}
        \captionsetup{skip=3pt}
        \centering
        \includegraphics{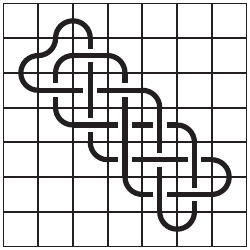}
        \caption*{\ka{11}{180}}
    \end{minipage} \hfill
     \begin{minipage}{0.17\textwidth}
        \captionsetup{skip=3pt}
        \centering
        \includegraphics{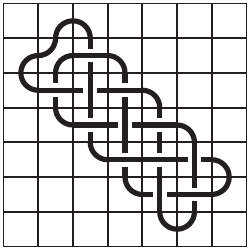}
        \caption*{\ka{11}{184}}
    \end{minipage} \hfill
    \begin{minipage}{0.17\textwidth}
        \captionsetup{skip=3pt}
        \centering
        \includegraphics{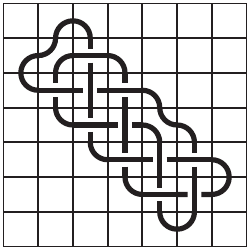}
        \caption*{\ka{11}{185}}
    \end{minipage} \hfill \phantom{.}
\end{figure}
\begin{figure}
    \vspace{-.1in}
    \centering
    \phantom{.} \hfill
    \begin{minipage}{0.17\textwidth}
        \captionsetup{skip=3pt}
        \centering
        \includegraphics{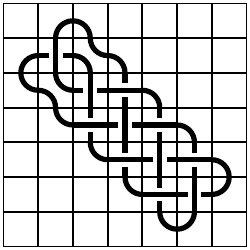}
        \caption*{\ka{11}{192}}
    \end{minipage} \hfill
     \begin{minipage}{0.17\textwidth}
        \captionsetup{skip=3pt}
        \centering
        \includegraphics{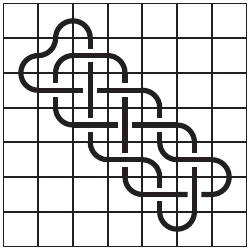}
        \caption*{\ka{11}{203}}
    \end{minipage} \hfill
    \begin{minipage}{0.17\textwidth}
        \captionsetup{skip=3pt}
        \centering
        \includegraphics{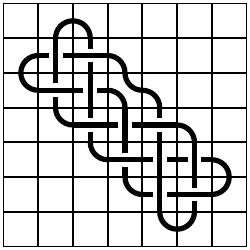}
        \caption*{\ka{11}{205}}
    \end{minipage} \hfill
     \begin{minipage}{0.17\textwidth}
        \captionsetup{skip=3pt}
        \centering
        \includegraphics{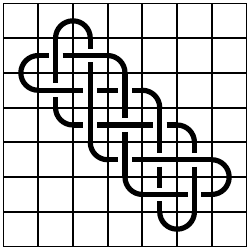}
        \caption*{\ka{11}{210}}
    \end{minipage} \hfill
    \begin{minipage}{0.17\textwidth}
        \captionsetup{skip=3pt}
        \centering
        \includegraphics{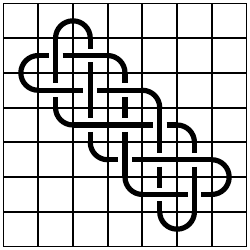}
        \caption*{\ka{11}{226}}
    \end{minipage} \hfill \phantom{.}
\end{figure}
\begin{figure}
    \vspace{-.1in}
    \centering
    \phantom{.} \hfill
    \begin{minipage}{0.17\textwidth}
        \captionsetup{skip=3pt}
        \centering
        \includegraphics{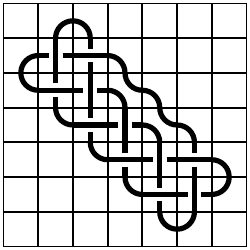}
        \caption*{\ka{11}{306}}
    \end{minipage} \hfill
     \begin{minipage}{0.17\textwidth}
        \captionsetup{skip=3pt}
        \centering
        \includegraphics{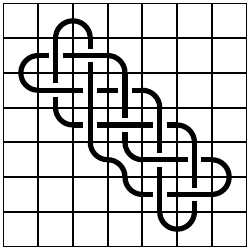}
        \caption*{\ka{11}{307}}
    \end{minipage} \hfill
    \begin{minipage}{0.17\textwidth}
        \captionsetup{skip=3pt}
        \centering
        \includegraphics{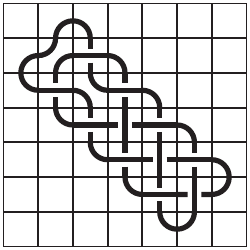}
        \caption*{\ka{11}{308}}
    \end{minipage} \hfill
     \begin{minipage}{0.17\textwidth}
        \captionsetup{skip=3pt}
        \centering
        \includegraphics{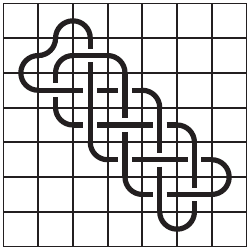}
        \caption*{\ka{11}{309}}
    \end{minipage} \hfill
    \begin{minipage}{0.17\textwidth}
        \captionsetup{skip=3pt}
        \centering
        \includegraphics{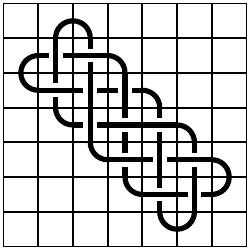}
        \caption*{\ka{11}{311}}
    \end{minipage} \hfill \phantom{.}
\end{figure}
\begin{figure}
    \vspace{-.1in}
    \centering
    \phantom{.} \hfill
    \begin{minipage}{0.17\textwidth}
        \captionsetup{skip=3pt}
        \centering
        \includegraphics{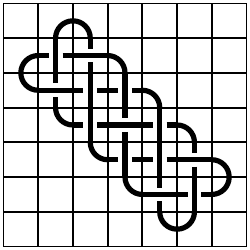}
        \caption*{\ka{11}{333}}
    \end{minipage} \hfill
     \begin{minipage}{0.17\textwidth}
        \captionsetup{skip=3pt}
        \centering
        \includegraphics{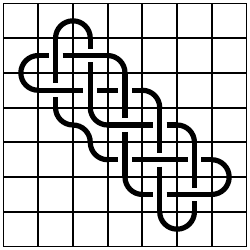}
        \caption*{\ka{11}{336}}
    \end{minipage} \hfill
    \begin{minipage}{0.17\textwidth}
        \captionsetup{skip=3pt}
        \centering
        \includegraphics{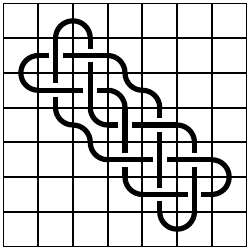}
        \caption*{\ka{11}{337}}
    \end{minipage} \hfill
     \begin{minipage}{0.17\textwidth}
        \captionsetup{skip=3pt}
        \centering
        \includegraphics{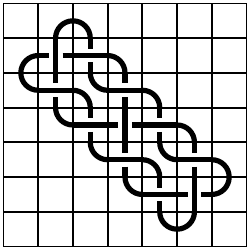}
        \caption*{\ka{11}{363}}
    \end{minipage} \hfill
    \begin{minipage}{0.17\textwidth}
        \captionsetup{skip=3pt}
        \centering
        \includegraphics{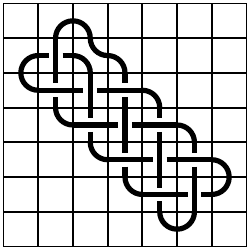}
        \caption*{\ka{12}{541}}
    \end{minipage} \hfill \phantom{.}
\end{figure}
\begin{figure}
    \vspace{-.1in}
    \centering
    \phantom{.} \hfill
    \begin{minipage}{0.17\textwidth}
        \captionsetup{skip=3pt}
        \centering
        \includegraphics{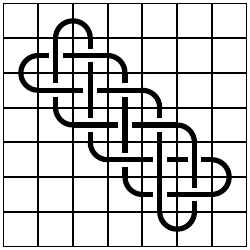}
        \caption*{\ka{12}{601}}
    \end{minipage} \hfill
     \begin{minipage}{0.17\textwidth}
        \captionsetup{skip=3pt}
        \centering
        \includegraphics{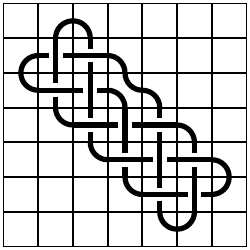}
        \caption*{\ka{12}{1024}}
    \end{minipage} \hfill
    \begin{minipage}{0.17\textwidth}
        \captionsetup{skip=3pt}
        \centering
        \includegraphics{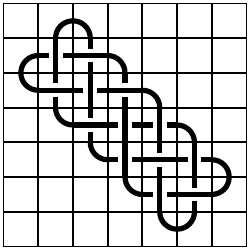}
        \caption*{\ka{12}{1034}}
    \end{minipage} \hfill
     \begin{minipage}{0.17\textwidth}
        \captionsetup{skip=3pt}
        \centering
        \includegraphics{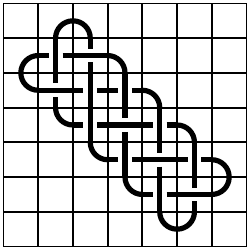}
        \caption*{\ka{12}{1126}}
    \end{minipage} \hfill
    \begin{minipage}{0.17\textwidth}
        \captionsetup{skip=3pt}
        \centering
        \includegraphics{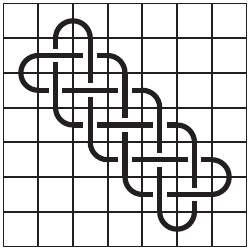}
        \caption*{\ka{13}{4304}}
    \end{minipage} \hfill \phantom{.}
\end{figure}

\bibliographystyle{amsplain}
\bibliography{bibliography}
\addcontentsline{toc}{section}{\refname}

\end{document}